\documentclass[11pt]{amsart}
\usepackage{amsmath,amssymb}

\begin{document}

\newcommand{\chp}{\mathds{1}}

\newtheorem{thm}{Theorem}[section]
\newtheorem{theorem}{Theorem}[section]
\newtheorem{lem}[thm]{Lemma}
\newtheorem{lemma}[thm]{Lemma}
\newtheorem{prop}[thm]{Proposition}
\newtheorem{proposition}[thm]{Proposition}
\newtheorem{cor}[thm]{Corollary}
\newtheorem{defn}[thm]{Definition}
\newtheorem*{remark}{Remark}
\newtheorem{conj}[thm]{Conjecture}
\newtheorem{objective}[thm]{Objective}


\newcommand{\Z}{{\mathbb Z}} 
\newcommand{\Q}{{\mathbb Q}}
\newcommand{\R}{{\mathbb R}}
\newcommand{\C}{{\mathbb C}}
\newcommand{\N}{{\mathbb N}}
\newcommand{\FF}{{\mathbb F}}
\newcommand{\fq}{\mathbb{F}_q}
\newcommand{\rmk}[1]{\footnote{{\bf Comment:} #1}}

\newcommand{\Gal}{\mathrm{Gal}}

\renewcommand{\mod}{\;\operatorname{mod}}
\newcommand{\ord}{\operatorname{ord}}
\newcommand{\norm}{\mathcal N} 
\newcommand{\intinf}{\int_{-\infty}^\infty}
\newcommand{\disc}{\operatorname{disc}}
\newcommand{\cond}{\operatorname{cond}} 
\newcommand{\lcm}{\operatorname{lcm}}
\newcommand{\leg}[2]{\left( \frac{#1}{#2} \right)}  
\newcommand{\Li}{\operatorname{Li}}

\newcommand{\sumstar}{\sideset \and^{*} \to \sum}
\newcommand{\prodstar}{{\sideset \and^{*} \to \prod}}

\newcommand{\sumf}{\sum^\flat}
\newcommand{\CF}{c_0} 
\newcommand{\Ht}{\operatorname{Ht}}

\newcommand{\sign}{\operatorname{sign}} 

\newcommand{\divid}{d} 

\newcommand{\pol}{\mathcal P_d}
\newcommand{\Bad}{\operatorname{Bad}}

\newcommand{\pols}{\mathcal P_\gamma(T)}

 \title[A lower bound on the LCM of polynomial sequences]{A lower bound on the least common multiple of polynomial sequences}
 \author{James Maynard and Ze'ev Rudnick}
 \date{\today}
 
 \begin{abstract} 
 For a polynomial $f\in \Z[x]$ with integer coefficients which is irreducible over the rationals of degree $d\geq 2$, Cilleruelo conjectured  that   the least common multiple   of the values of the polynomial at the first $N$ integers satisfies $\log \lcm (f(1),\dots, f(N)) \sim (d-1) N\log N$ as $N\to \infty$. This is only known for degree $d=2$. We give a lower bound for all degrees $d\geq 2$ which is consistent with the conjecture:  $\log \lcm (f(1),\dots, f(N)) \gg N\log N$.  
 \end{abstract}
 
 \address{Mathematical Institute, Radcliffe observatory quarter, Woodstock Road, Oxford OX2 6GG,
England}
\email{james.alexander.maynard@gmail.com}

 \address{Raymond and Beverly Sackler School of Mathematical Sciences,
Tel Aviv University, Tel Aviv 69978, Israel}
\email{rudnick@tauex.tau.ac.il}

  \thanks{This project has received funding from the European Research Council (ERC) under the European Union's Horizon 2020 research and innovation programme 
 (grant agreement ${\rm n}^{o}$ 786758).
}
 \maketitle
 
\section{The LCM problem}
For a polynomial $f\in \Z[X]$ with integer coefficients, set 
$$L_f(N):=\lcm\{f(n):n=1,\dots,N\}.
$$
The goal is to understand the asymptotic growth of $\log L_f(N)$ as $N\to \infty$. 

It is a well known and elementary fact that the least common multiple of all integers $1,2,\dots,N$ is exactly given by 
$$\log \lcm\{1,2,\dots,N\} = \psi(N):=\sum_{n\leq N} \Lambda(n)$$
with $\Lambda(n)$ being the von Mangoldt function, and hence by the Prime Number Theorem, 
$$
\log \lcm\{1,2,\dots,N\} \sim N .
$$
A similar growth occurs for products of linear polynomials \cite{HQT}.  
  
  However, in the case of irreducible polynomials higher degree, 
Cilleruelo \cite{Cilleruelo} conjectured that the growth is faster than linear, precisely: 
\begin{conj}\label{conj cil} 
If $f$ is an irreducible polynomial with $\deg f\geq 2$, then
$$\log L_f(N)\sim (\deg f-1)N\log N,\quad N\to \infty .
$$
\end{conj}
 Cilleruelo proved Conjecture~\ref{conj cil} for quadratic polynomials. 
No other case of Conjecture~\ref{conj cil} is known to date.   
\begin{remark}
An examination of Cilleruelo's argument shows that for any irreducible $f$ of degree $d\geq 3$, we have an upper bound
\[
\log L_f(N)\lesssim (d -1) N\log N .  
\]     
Here $f\lesssim g$ means that $|f(x)|\le (1+o(1))g(x)$.
 \end{remark}
In this note, we give a lower bound of the right order of magnitude: 

\begin{thm}\label{prop:lower bound}
Let $f\in \Z[x]$ be irreducible,  of degree $d\geq 2$. Then  
\[
\log L_f(N) \gg N\log N .
\]
\end{thm}
\begin{remark}
The argument gives that $\log L_f(N) \gtrsim \frac {1}{d}N\log N$.
\end{remark}

\begin{cor}
Suppose $f\in \Z[x]$ has an irreducible factor of degree $\geq 2$, i.e. $f(x)$  is not a product of linear polynomials (over $\Q$). Then 
\[
 N\log N\ll  \log L_f(N)\ll N\log N.
 \] 
\end{cor}
This is because $\max (\lcm\{a_n\} ,\lcm\{b_m\}) \leq \lcm\{a_nb_n\}\leq \lcm\{a_n\} \cdot \lcm\{b_m\} $. 

\bigskip

Prior to this note, the only available bound was of size $\gg N$:  Hong {\em et al} \cite{HLQW} show that $\log L_f(N)\gg N$ for any polynomial with {\em non-negative} integer coefficients.   

\section{Proof of Theorem~\ref{prop:lower bound}}
Let $P^+(n)$ denote the largest prime factor of $n$. We will need  a  result on the greatest prime factor $P^+(f(n))$ of $f(n)$ (``Chebyshev's problem"). This is a well-studied subject, and we need a relatively simple bound, which we state here and explain in \S~\ref{appendix}: 
\begin{thm}
\label{lem:Nagel} 
 Let $f(x)\in \Z[x]$ be irreducible  of degree $d\geq 2$. Then 
\[
P^+(f(n)) > n 
\]
for a positive proportion  of integers $n$.  
\end{thm}
\begin{remark}
In fact one can show $P^+(f(n))>n$ for a proportion at least $1-\frac 1d$ of integers $n$.
\end{remark}

A result of this form goes back to T. Nagell in 1921 \cite{Nagell 1921}, though he did not state this with positive density, but instead with a better bound of $n(\log n)^a$ for all $a<1$. Once one gets a positive density, one automatically obtains a better bound of $n(\log n)$, again in a set of positive density, see \S~\ref{appendix}. 
A form of Theorem~\ref{lem:Nagel}  was given by Cassels \cite{Cassels} in 1960. The problem was studied by Erdos \cite{Erdos} in 1952, and in 1990 Tenenbaum \cite{Tenenbaum} showed that $P^+(f(n))>n\exp((\log n)^a)$ infinitely often for all $a<2-\log 4$. 

Alongside Theorem \ref{lem:Nagel}, we need the following simple lemma. Let 
\[
\mathcal{N}:=\Bigl\{n\in \Bigl[\frac{N}{\log{N}},N\Bigr]:\,P^+(f(n))>n\Bigr\}.
\]
  \begin{lem}\label{lem:root bound}
  Given a prime $p$, and for $N$ sufficiently large in terms of $f$, 
  the number of $m\in \mathcal N$  with $P^+(f(m))=p$ is at most $d$.
  \end{lem}
   \begin{proof}
If $P^+(f(m))=p$ then we must have
\[
f(m)\equiv 0\pmod{p}.
\]
If $m\in\mathcal{N}$ and $P^+(f(m))=p$ we must also have that $N/\log{N}\le m<p$. Since $p>N/\log{N}$ and $N$ is sufficiently large in terms of $f$, we see that $f$ is a non-zero polynomial modulo $p$. Therefore $f$ has at most $d$ roots modulo $p$, and all choices of $m$ must be congruent to one of these roots. Since we only consider $0<m<p$, there is at most one choice of $m\equiv a\pmod{p}$ for each root $a$ modulo $p$, and so at most $d$ choices of $m$.
\end{proof}

\subsection{Proof of Theorem~\ref{prop:lower bound}}
  Given Theorem~\ref{lem:Nagel}, we proceed as follows. The result is trivial for bounded $N$, so we may assume that $N$ is sufficiently large in terms of $f$. 
 By Theorem~\ref{lem:Nagel}, there is an absolute constant $c>0$ such that $P_n>n$ for  $\gtrsim c N$ integers in $[1,N]$, and so certainly $\#\mathcal{N}\gtrsim c N$. Let
\[
\mathcal{P}:=\{P^+(f(n)):\,n\in\mathcal{N}\}
\]
be the set of largest prime factors occurring. Then, by Lemma \ref{lem:root bound}, we have that
\[
c N\lesssim \#\mathcal{N}=\sum_{p}\#\{n\in\mathcal{N}:\,P^+(f(n))=p\}\le d\#\mathcal{P},
\]
and so
\[
\#\mathcal{P}\gtrsim \frac{c N}{d}.
\]
Moreover, by definition of $\mathcal{N}$, if $p\in\mathcal{P}$ then $p>N/\log{N}$ and $p|f(n)$ for some $n\le N$. Therefore we have that
  \[
  \begin{split}
  \log \lcm (f(1),\dots, f(N))  &\geq \sum_{p\in \mathcal{P} }  \log p \geq \#\mathcal{P} \log\frac{N}{\log{N}}
  \\
  &
  \gtrsim \frac {c N}{d }\log N,
  \end{split}
  \]
  as claimed. \qed
  
  \section{Proof of Theorem~\ref{lem:Nagel} } \label{appendix}
  We begin by recording a simple bound on the number of times a prime $p$ can divide values of $f$. Let $\alpha_p(N)$ be the exponents in the prime factorization
  \[
  \prod_{n=1}^N|f(n)|=\prod_p p^{\alpha_p(N)}.
  \]
  We then have the following result.
 \begin{lem}\label{lem:alpha bounds}
Let $\rho_f(m)$ denote the number of roots of $f$ modulo $m$. Assume that $f$ has no rational zeros. Let $p$ be a prime, $p\leq N$. 

Then if $p\nmid \disc f$, we have 
 \begin{equation}
\alpha_p(N) = N\frac{\rho_f(p)}{p-1} + O(\frac{\log N}{\log p}) 
 \end{equation}
and if $p\mid \disc f$, we have
\[
\alpha_p(N)\ll \frac{N}{p} . 
\]
\end{lem} 

\begin{proof}
Since $f$ has no rational zeros, $\prod_{n=1}^Nf(n)\neq 0$ and so $\alpha_p(N)$ is well defined.  By definition, 
\[
\alpha_p(N) = \sum_{n\leq N} \sum_{k\geq 1} \mathbf{1}(p^k\mid f(n))=\sum_{1\leq k\lesssim \frac{d\log N}{\log p} } \#\{n\leq N: f(n)=0\bmod p^k\} . 
\]
To count the number $\#\{n\leq N: f(n)=0\bmod p^k\}$, divide the interval $[1,N]$ into $\lfloor N/p^k \rfloor$ consecutive intervals of length $p^k$, and a remaining interval. On each such interval of length $p^k$, the number of solutions of $f(n)=p^k$ is the total number $\rho_f(p^k)$ of solutions of this congruence. On the remaining interval, the number of solutions is not greater than that. 
Hence
 \[
\alpha_p(N) =\sum_{1\leq k\lesssim \frac{d\log N}{\log p} } \rho_f(p^k) \Bigl(\lfloor \frac{N}{p^k} \rfloor+O(1)\Bigr) . 
\]
By Hensel's lemma,  $\rho_f(p^k)=\rho_f(p)$ for $p\nmid \disc f$.  Hence for $p\nmid \disc f$
 \[
\alpha_p(N) =\sum_{1\leq k\lesssim \frac{d\log N}{\log p} } \rho_f(p)\Bigl (\lfloor \frac{N}{p^k} \rfloor+O(1)\Bigr) 
=\rho_f(p)\Big(\frac{N}{p-1} +O(\frac{\log N}{\log p}) \Big) . 
\]

For primes $p\mid \disc f$ dividing the discriminant of $f$, a more detailed examination gives the bound \cite[Th\'eor\`eme II]{Nagell 1921}
\[
\rho_f(p^k) \leq d (\disc f)^2 =O(1)
\]
which gives for $p\mid \disc f$
\[
\alpha_p(N)\ll_f \sum_{1\leq k\lesssim \frac{d\log N}{\log p} }  \Bigl(\lfloor \frac{N}{p^k} \rfloor+O(1)\Bigr) \ll  \frac Np,
\]
 as claimed. 
 \end{proof} 
  
\begin{proof} 
Let $N_-:= N/\log N$, and define  the   exceptional set $\mathcal E(N)\subseteq (N_-,N] $ by 
\[ 
 \mathcal E(N):=\{ N_-<n\leq N:  P^+(f(n)) \leq n \} .
\]
Let 
\[
Q(N):=\prod_{n\in \mathcal E(N)} |f(n)| . 
\]
We compute $\log Q(N)$ in two ways:

Using $\log |f(n)| \sim d\log n$ as $n\rightarrow \infty$, we have
\[
\log Q(N) = \sum_{ n\in \mathcal E(N)} \log |f(n)|  \sim \sum_{ n\in \mathcal E(N)}   d\log n  .
\]
Since $\log n\sim \log N$  for $n\in \mathcal E(N)\subseteq [N_-,N]$, we have  
\[
 \sum_{ n\in \mathcal E(N)}  d\log n \sim  d\log N \#\mathcal E(N) 
\]
so that 
\begin{equation}\label{lower bound for log Q} 
\log Q(N)\sim  d\log N \#\mathcal E(N)  .
\end{equation}

On the other hand, write the prime power decomposition of $Q(N)$ as
\[
Q(N) =\prod_{ n\in \mathcal E(N)}|f(n)|= \prod_ p  p^{\gamma_p(N)} . 
\]
Since $P^+(f(n))\leq n \leq N$ 
for all $n\in \mathcal E(N)$, only primes $p\leq N$ appear in the product. Thus
\[
\log Q(N) =\sum_{p\leq N} \gamma_p(N) \log p . 
\]
We also have $\gamma_p(N)\leq \alpha_p(N)$ where $\prod_{n=1}^N |f(n)| = \prod_p p^{\alpha_p(N)}$. 
Thus
\[
\log Q(N)\leq \sum_{p\leq N} \alpha_p(N) \log p . 
\]
Therefore, by Lemma \ref{lem:alpha bounds},
 \begin{equation*}
 \begin{split}
 \log Q(N) &\leq \sum_{p\leq N} \alpha_p(N) \log p
 \\ 
& \leq \sum_{p\leq N}\Big(N\frac{\rho_f(p)}{p-1} + O(\frac{\log N}{\log p}) \Big)\log p +O(\sum_{p\mid \disc f}\frac{N\log p}{p})
\\
& = N\sum_{p\leq N} \frac{\rho_f(p)\log p}{p-1} + O\Big(\pi(N) \log N\Big) +O(N) . 
\end{split}
\end{equation*}
Now for $f$ irreducible it follows from the Chebotarev density theorem (or earlier work of Kronecker or Frobenius) that (see \cite[equation (4)]{Nagell 1921}): 
\[
\sum_{p\leq N} \frac{\rho_f(p)\log p}{p-1}=\log N+O(1),
\]
 hence
\[
\log Q(N)   \leq  N\Big(\log N+O(1)\Big) +O(N) \sim N\log N . 
\]
  Comparing with \eqref{lower bound for log Q} gives
\[
  d \log N \#\mathcal E(N) \lesssim \log Q(N) \lesssim  N\log N
  \]
  and hence we obtain
  \[
  \#\mathcal E(N) \lesssim \frac 1d N . 
  \]
  Therefore
  \[
  \#\{n\in [1,N]: P^+(f(n))<n \} \leq N_-  +\#\mathcal E(N) \lesssim N_-+ \frac 1d N\lesssim \frac 1d N , 
  \]
that is the proportion of elements of $[1,N]$ with $P^+(f(n))<n$ is at most $1/d$. 
  \end{proof}
  
 We owe to Andrew Granville the following observation: Theorem~\ref{lem:Nagel} can be boot-strapped to give a slightly better result:
  \begin{cor}
   Let $f(x)\in \Z[x]$ be irreducible  of degree $d\geq 2$. Then for any $\delta<1/d^2$, 
\[
P^+(f(n)) > \delta n \log n   
\]
for a positive proportion  
of the integers.  
  \end{cor}
   
  \begin{proof}
Let $\delta>0$ be fixed, and let
\[
\mathcal S:=\Bigl\{n\in \Bigl[\frac{N}{\log{N}},N\Bigr]:\,P^+(f(n))<\delta n\log n\} .
\]
 Assume by contradiction that $\mathcal S$ has full density, that is $\#\mathcal S\sim N$ as $N\rightarrow \infty$. As before, let 
    \[
    \mathcal N := \{\frac N{\log N}<n\leq N: P^+(f(n))>n\}.
     \] 
 We saw that $\#\mathcal{N}\gtrsim \frac{1}{d} N$. Since $\#\mathcal S\sim N$ has density one by assumption we see that $\#\mathcal N\cap \mathcal S\gtrsim \frac 1d N$.  Let
     \[
     \mathcal P_\mathcal{S}:=\{ P^+(f(n)): n\in \mathcal S\cap \mathcal N\}
     \]
     be the set of largest prime divisors arising from $n\in \mathcal N\cap \mathcal S$. 
     Then we saw that each prime $p\in \mathcal P_N$ can occur at most $d$ times as some $P^+(f(m))$ for $m\in \mathcal N$, and so 
     \[
     \#\mathcal P_\mathcal{S}\geq \frac 1{d} \#\mathcal N\cap \mathcal S \gtrsim \frac 1{d^2} N . 
     \]

     On the other hand, since $P^+(f(n))<\delta n\log n $ for $n\in \mathcal S\cap \mathcal N$, we must have 
$     \mathcal P_\mathcal{S}\subseteq [1,\delta N\log N]$. 
     Therefore 
     \[
     \#\mathcal P_\mathcal{S}\leq \pi(\delta N\log N) \sim \delta N
     \]
     by the Prime Number Theorem. Thus 
     \[
     \frac 1{d^2} N \lesssim \#\mathcal P_\mathcal{S} \lesssim \delta N
     \]
     which is a contradiction if $\delta<1/d^2$. 
  \end{proof}

 \end{document}